\documentclass[11pt]{article}
\usepackage{article}
\author{Torgeir Aamb\o}
\title{Classification of localizing subcategories along $t$-structures}
\date{}

\begin{document}
\color{myblack}
\maketitle

\begin{abstract}
    We study the interplay between localizing subcategories in a stable $\infty$-category $\C$ with $t$-structure $(\C\geqz, \C\leqz)$, the prestable $\infty$-category $\C\geqz$ and the abelian category $\C^\heart$. We prove that weak localizing subcategories of $\C^\heart$ are in bijection with the localizing subcategories of $\C$ where object-containment can be checked on the heart. This generalizes similar known correspondences for noetherian rings and bounded $t$-structures. We also prove that this restricts to a bijection between localizing subcategories of $\C^\heart$, and localizing subcategories of $\C$ that are kernels of $t$-exact functors --- lifting Lurie's correspondence between localizing subcategories in $\C\geqz$ and $\C^\heart$ to the stable category $\C$. 
\end{abstract}

{\hypersetup{linkcolor=myblack}\tableofcontents}

\section{Introduction}

The concept of a $t$-structure on a triangulated category was introduced in \cite{beilinson-bernstein-deligne_1983}, and in a way axiomatizes the concept of taking the homology of a chain complex in the derived category of a ring. Most interesting triangulated categories arise as the homotopy category of a stable $\infty$-category, and the concept of a $t$-structure lifts to this setting. Having a $t$-structure allows us to naturally compare features of a stable $\infty$-category $\C$ to features of an abelian category $\C^\heart$, called the heart of the given $t$-structure. 

In order to understand the internal structure of a stable $\infty$-category, is its important to understand its \emph{localizing subcategories}. A full subcategory is called localizing if it is a stable full subcategory closed under colimits. The goal of this paper is to classify the localizing subcategories of $\C$ that interact well with $t$-structures. These are the localizing subcategories $\L\subseteq \C$ that inherit a $t$-structure, and you can check if an object $X$ is in $\L$ by checking whether $\pi_n^\heart X\in \L^\heart$. We call these the \emph{$\pi$-stable localizing subcategories} --- see \cref{def:pi-stable-localizing-subcategory}. 

We want to compare these localizing subcategories of $\C$ to subcategories of $\C^\heart$. The abelian analog of localizing subcategories of a stable $\infty$-category, are the \emph{weak Serre subcategories} closed under coproducts. We call these the \emph{weak localizing subcategories}. Our first main result is the following classification of $\pi$-stable localizing subcategories in $\C$ via the heart construction. This generalizes a similar correspondence due to Takahashi (\cite{takahashi_2009}) for commutative noetherian rings, see \cref{cor:takahashi-weak-localizing}. 

\begin{introthm}[\cref{thm:premain}]
    \label{thm:A}
    Let $\C$ be a stable $\infty$-category with a $t$-structure. If the $t$-structure is right complete and compatible with filtered colimits, then the map $\L\longmapsto \L^\heart$ gives a one-to-one correspondence between $\pi$-stable localizing subcategories of $\C$ and weak localizing subcategories in $\C^\heart$.
\end{introthm}

The above theorem also holds when we exclude the existence of coproducts, giving a one-to-one correspondence between $\pi$-stable thick subcategories of $\C$ and weak Serre subcategories of $\C^\heart$. This generalizes the similar result of Zhang--Cai (\cite{zhang-cai_2017}) to the setting of unbounded $t$-structures, see \cref{prop:classification-weak-serre} and \cref{cor:classification-weak-serre-bounded}. 

We also want a way to study the analog of (non-weak) Serre subcategories of $\C^\heart$ closed under coproducts --- called the \emph{localizing subcategories} of $\C^\heart$ --- in the stable $\infty$-category $\C$. In order to do this we use the bridge between stable $\infty$-categories with a $t$-structure and prestable $\infty$-categories, as developed mainly by Lurie in \cite[App. C]{lurie_SAG}. A prestable $\infty$-category acts as the connected part of the $t$-structure, denoted $\C\geqz$, and they allow us to study $t$-structures on $\C$ indirectly, without carrying around extra data. 

Lurie introduced the notion of localizing subcategories of the prestable $\infty$-category $\C\geqz$, which more closely mimics the construction of localizing subcategories of abelian categories. The analog of $\pi$-stable localizing subcategory in this situation are called \emph{separating localizing subcategories} by Lurie. Using the heart construction for prestable $\infty$-categories, Lurie classified the separating localizing subcategories of $\C\geqz$ in \cite[C.5.2.7]{lurie_SAG}, by proving that there is a one-to-one correspondence
\[\separating \simeq \ablocalizing.\]
Our second main theorem provides an extension of this correspondence to the stable $\infty$-category $\C$, allowing us to strengthen \cref{thm:A} to non-weak localizing sucategories. This interacts well with existing classifications of localizing subcategories in modules over noetherian rings and quasicoherent sheaves on noetherian schemes. 


\begin{introthm}[{\cref{thm:main}}]
    \label{thm:B}
    Let $\C$ be a stable category with a $t$-structure. If the $t$-structure is right complete and compatible with filtered colimits, then the map $\L\longmapsto \L^\heart$ gives a one-to-one correspondence between localizing subcategories of $\C^\heart$, and $\pi$-stable localizing subcategories of $\C$ that are kernels of a $t$-exact localization.
\end{introthm}

Note that any stable $\infty$-category is prestable, hence the above result might at first glance seem to follow trivially from Luries's classification. But, any separating localizing subcategory of a stable $\infty$-category $\C$, viewed as a prestable one, is the whole category $\C$ by \cite[C.1.2.14, C.5.2.4]{lurie_SAG}. This means that the stable situation needs its own separate treatment, hence the existence of the current paper. 

The results of the paper can be summarized in the following diagram, showcasing the bijections ($\simeq$) and the inclusions ($\subseteq$) between the different types of subcategories. 

\begin{center}
    \begin{tikzcd}
        \tcompatible 
        \arrow[rrd, "(-)^\heart"]
        &&\\
        \pistable 
        \arrow[rr, "\simeq"] 
        \arrow[u, hook, "\subseteq", swap] 
        && 
        \weaklocalizing  
        \\
        \piexact 
        \arrow[r, "\simeq"] 
        \arrow[d, hook, "\subseteq"]
        \arrow[u, hook, "\subseteq", swap] 
        & 
        \separating 
        \arrow[r, "\simeq"]
        \arrow[d, hook, "\subseteq"]
        & 
        \ablocalizing
        \arrow[u, hook, "\subseteq", swap]
        \\
        \texact 
        \arrow[r, "(-)\geqz"]
        & 
        \prlocalizing 
        \arrow[ru, "(-)\leqz", swap] 
        &                                
        \end{tikzcd}
\end{center}

\textbf{Linear overview:} We start \cref{sec:prestable-and-stable-categories} with some recollections on $t$-structures, prestable $\infty$-categories, and their interactions, before we introduce the notion of localizing subcategories in \cref{ssec:localizing-subcategories}. We then study some further interactions between these, which we use to prove \cref{thm:A} in \cref{ssec:classificartion-weak-localizing} and \cref{thm:B} in \cref{ssec:classificartion-localizing}. We finish the paper by looking at some consequences and applications of our results. 

\textbf{Conventions:} We will work in the setting of $\infty$-categories, as developed by Lurie in \cite{lurie_09} and \cite{Lurie_HA}. We will restrict our attention to presentable stable $\infty$-categories, which we will just call \emph{stable categories}. Given a stable category $\C$ with a nice $t$-structure, its associated prestable category will be denoted $\C\geqz$ and its heart by $\C^\heart$. We assume all $t$-structures to be accessible.

\textbf{Acknowledgements:} We wish to thank Drew Heard and Marius Nielsen for helpful conversations. This work was partially finished during the authors visit to the GeoTop center at the University of Copenhagen, which we gratefully thank for their hospitality. This work was supported by grant number TMS2020TMT02 from the Trond Mohn Foundation.    

\section{Prestable and stable categories}
\label{sec:prestable-and-stable-categories}

For the rest of the paper we fix a stable category $\C$. We wish to equip this with a $t$-structure, which will allow us to always have a comparison from $\C$ to an abelian category. The main reference for $t$-structures in this setting is \cite[Sec 1.2.1]{Lurie_HA}. Note that, as opposed to much of the homological algebra literature, we follow Lurie's homological indexing convention.  

\begin{definition}
    A \emph{$t$-structure} on $\C$ is a pair of full subcategories $(\C\geqz, \C\leqz)$ such that:
    \begin{enumerate}
        \item The mapping space $\Map_\C(X,Y[-1])\simeq 0$ for all $X\in \C\geqz$ and $Y\in \C\leqz$;
        \item There are inclusions $\C\geqz[1]\subseteq \C\geqz$ and $\C\leqz[-1]\subseteq \C\leqz$;
        \item For any $Y\in \C$ there is a fiber sequence $X\to Y\to Z$ such that $X\in \C\geqz$ and $Z[1]\in \C\leqz$. 
    \end{enumerate} 
\end{definition}

This is equivalent to choosing a $t$-structure on the homotopy category $h\C$, which is a triangulated category. Hence the contents of this paper should be equally useful to those familiar with $t$-structures on triangulated categories. 

We will assume all $t$-structures to be accessible, in the sense that the connected part $\C\geqz$ is presentable. By \cite[1.2.16]{Lurie_HA} the inclusions $\C\geqz\to \C$ and $\C\leqz\to \C$ have a right adjoint $\tau\geqz$ and a left adjoint $\tau\leqz$ respectively. We denote $\C_{\geq n} := \C\geqz[n]$ and $\C_{\leq n} := \C\leqz[n]$. 

\begin{definition}
    The heart of a $t$-structure $(\C\geqz, \C\leqz)$ on $\C$ is defined as the full subcategory $\C^\heart := \C\geqz\cap\C\leqz$.
\end{definition}

The heart $\C^\heart$ is always equivalent to the nerve of its homotopy category $h\C^\heart$, which was proven in \cite{beilinson-bernstein-deligne_1983} to be an abelian category. It is standard to follow \cite[1.2.1.12]{Lurie_HA} and identify the two. 

\begin{definition}
    The composite functor $\tau\geqz\circ\tau\leqz\simeq\tau\leqz\circ\tau\geqz \: \C\to \C^\heart$ is denoted by $\pi_0^\heart$ and its composition with the shift functor $X\to X[-n]$ by $\pi_n^\heart$. These are called the \emph{heart-valued homotopy groups} of $X$. 
\end{definition}

The last definition we will need, before going on to prestable categories is the following niceness condition. 

\begin{definition}
    A $t$-structure $(\C\geqz, \C\leqz)$ on a stable category $\C$ is said to be \emph{compatible with filtered colimits} if $\C\leqz$ is closed under all filtered colimits in $\C$. 
\end{definition}

We now recall the notion of prestable $\infty$-categories, which, similarily to the stable $\infty$-categories, we will simply call \emph{prestable categories}. The theory of prestable categories was developed by Lurie in \cite[App. C]{lurie_SAG}, and has since been applied in a varied range of areas. We define these as follows. 

\begin{definition}
    An $\infty$-category $\D$ is \emph{prestable} if there exists a stable category $\C$ with a $t$-structure $(\C\geqz, \C\leqz)$, such that $\D\simeq \C\geqz$.
\end{definition}

\begin{remark}
    This is not the most general, nor the standard, definition of a prestable category --- see \cite[C.1.2.1]{lurie_SAG} --- but by \cite[C.1.2.9]{lurie_SAG} the above definition describes all prestable categories admitting finite limits, hence it is not a very severe restriction. The category $\D$ is also not unique, see \cite[C.1.2.10]{lurie_SAG}, but we will mostly focus on the choice 
    \[\D=\Sp(\C\geqz) = \colim (\cdots\overset{\Omega}\to\C\geqz\overset{\Omega}\to\C\geqz).\]
\end{remark}

Since we will discuss both stable and prestable categories, and their interactions, we will try to consequently denote prestable categories by $\C\geqz$ and stable categories by $\C$. 

\begin{remark}
    \label{rm:stable-is-prestable}
    Any stable category $\C$ is prestable, as seen by choosing the trivial $t$-structure $(\C,0)$. This is both a blessing, as it allows us to talk about both in a common language, and a curse, as using common language can be rather confusing when trying to study their interactions.
\end{remark}

We will restrict our attention to Grothendieck prestable categories, which are prestable categories that work well with colimits. There are numerous different equivalent definition of these, see \cite[C.1.4.1]{lurie_SAG}, but the one best related to the above definition of a prestable category is the following. 

\begin{definition}
    A prestable category $\C\geqz$ is \emph{Grothendieck} if the $t$-structure on its associated stable category $\C$ is compatible with filtered colimits. 
\end{definition}

The following example is perhaps the main reason for the naming convention.

\begin{example}
    For any Grothendieck abelian category $\A$, the derived category category $\D(\A)$ has a natural $t$-structure with heart $\A$. The connected component $\D(\A)\geqz$, which consists of complexes $X_\bullet$ such that $H_i(X_\bullet) = 0$ for $i<0$ is a Grothendieck prestable category.  
\end{example}

We also have some examples showing up in stable homotopy theory. 

\begin{example}
    Let $\Sp$ be the stable $\infty$-category of spectra. This has a natural $t$-structure with heart $\Ab$. The connected component $\Sp\geqz$, consisting of connective spectra, is a Grothendieck prestable category. 
\end{example}

\begin{example}
    Important for modern homotopy theory is the category of $E$-based synthetic spectra $\SynE$ for some Landweber exact homology theory $E$, see \cite{pstragowski_2022}. This has a naturally occurring $t$-structure with heart $\Comod\EE$, and its connected component $\SynE^{\geq 0}$ is Grothendieck prestable. This example is one of our main motivations for this work, and we plan to study the applications of the contents in this paper to synthetic spectra in future work. 
\end{example}



\begin{remark}
    \label{rm:stable-is-grothendieck-prestable}
    If the prestable category $\C\geqz$ is compactly generated, then it is automatically Grothendieck, see \cite[C.1.4.4]{lurie_SAG}. A stable $\infty$-category $\C$ is, as mentioned above, also prestable. It is in fact Grothendieck if and only if it is presentable. 
\end{remark}

\begin{definition}
    We say a $t$-structure on a stable category $\C$ is \emph{right complete} if the natural functor $\displaystyle {\underset{n}\colim} \C_{\geq -n} \overset{\simeq}\to \C $ is an equivalence. 
\end{definition}

\begin{remark}
    \label{rm:grothendieck-iff-right-complete-and-colims}
    For any Grothendieck prestable category $\C\geqz$ the functor $\Sp(-)$, sending $\C\geqz$ to its stabilization, $\Sp(\C\geqz)$, provides a one-to-one correspondence between Grothendieck prestable categories and stable categories equipped with a right complete $t$-structure compatible with filtered colimits. This is one of the main reasons to study prestable categories, as being prestable is a property, while having a $t$-structure is extra structure. 
\end{remark}

\begin{remark}
    If $\C$ is a stable category with a $t$-structure compatible with filtered colimits, then the heart-valued homotopy groups functors $\pi_n^\heart$ preserve filtered colimits. 
\end{remark}

\subsection{Bridging the gap}

In this section we study the passage from stable to prestable and vice versa. In particular we look into when they determine each other. 

If $\C$ is a stable category with a right complete $t$-structure $(\C\geqz, \C\leqz)$, we can reconstruct it from its connected component. 

\begin{lemma}[{\cite[C.1.2.10]{lurie_SAG}}]
    \label{lm:right-complete-then-equiv-to-sp}
    Let $\C$ be a stable category. If $\C$ has a right complete $t$-structure, then there is an equivalence $\Sp(\C\geqz)\simeq \C$. 
\end{lemma}

This fact also extends to equivalences of categories, as proven by Antieau. 

\begin{lemma}[{\cite[6.1]{antieau_2021}}]
    \label{lm:if-prestable-equiv-then-stable-equiv}
    Let $\C$ and $\D$ be stable categories equipped with right complete $t$-structures. If $\C\geqz\simeq \D\geqz$, then also $\C\simeq \D$. 
\end{lemma}

\begin{remark}
    In particular both the above results hold for any $\C$ such that $\C\geqz$ is Grothendieck. 
\end{remark}

We can also naturally go in the other direction. If we have an equivalence of stable categories $\C\simeq \D$, that is compatible with the $t$-structures, then we get an induced equivalence on the connected components. The precise definition of being compatible with the $t$-structures is as follows. 

\begin{definition}
    Let $\C, \D$ be stable categories with $t$-structures. An exact functor $F\:\C\to\D$ is \emph{right $t$-exact} if $F(\C\geqz)\subseteq \D\geqz$. The notion of \emph{left $t$-exactness} is defined similarly. If $F$ satisfies both, we say that it is a \emph{$t$-exact functor}. 
\end{definition}

\begin{remark}
    This convention might seem wrong to readers with a background in homological algebra, as the role of left and right $t$-exact functors are usually the opposite. This flip is a consequence of using the homological indexing convention rather than cohomological indexing. 
\end{remark}

The above can then be made precise as follows. 

\begin{lemma}
    Let $\C, \D$ be stable categories with $t$-structures. If $F\:\C\to\D$ is a right $t$-exact functor, then we have an induced functor of prestable categories $F\geqz\:\C\geqz\to\D\geqz$. If $F$ is an equivalence, then so is $F\geqz$. 
\end{lemma}

For the rest of the paper we will use the following terminology. 

\begin{definition}
    A \emph{$t$-stable category} is a stable category $\C$ together with a choice of a right complete $t$-structure compatible with filtered colimits. 
\end{definition}

\begin{example}
    Let us see some examples of $t$-stable categories. 
    \begin{enumerate}
        \item For every commutative noetherian ring $R$, the derived category $\Der(R)$ together with its natural $t$-structure, is a $t$-stable category. 
        \item The category of spectra, together with its natural $t$-structure, is a $t$-stable category.
        \item The category of synthetic spectra, $\SynE$, together with its natural $t$-structure is a $t$-stable category. 
        \item For a noetherian scheme $X$, its associated derived category of quasi-coherent $\mathcal{O}_X$-modules, $\Der_{qc}(X)$, is $t$-stable. 
    \end{enumerate}
\end{example}

\begin{remark}
    Let $\C$ be a $t$-stable category. By definition we have that the connective part, $\C\geqz$, is a Grothendieck prestable category, and that the heart $\C^\heart$ is a Grothendieck abelian category. Hence $t$-stable categories serve as a natural place to study the interactions between these three types of categories. 
\end{remark}

\begin{remark}
    In \cite[Section C.3.1]{lurie_SAG} Lurie constructs a category of $t$-stable categories. If we denote this by $t\Cat$ then the contents of \cref{rm:grothendieck-iff-right-complete-and-colims} can be described as an adjoint pair of equivalences
    \begin{center}
        \begin{tikzcd}
            \Groth \arrow[r, "\Sp(-)", yshift=2pt] & t\Cat \arrow[l, "(-)\geqz", yshift=-2pt].
        \end{tikzcd}
    \end{center}
    This should, however, be viewed as a heuristic rather than a very precise statement, as the right hand category is a bit tricky to define. 
\end{remark}

\subsection{Localizing subcategories}
\label{ssec:localizing-subcategories}

We now turn our attention to localizing subcategories. As we are working in three interconnected settings --- stable, prestable and abelian --- and all settings use the same terminology, we feel that this section is very ripe for confusions to occur . In an attempt to clarify which setting we are in, we will usually refer to localizing subcategories of stable categories as \emph{stable localizing subcategories}, localizing subcategories of prestable categories as \emph{prestable localizing subcategories} and localizing subcategories of abelian categories as \emph{abelian localizing subcategories}. We will, however, sometimes omit the categorical prefix when we feel that it is clear from context. 

\begin{definition}
    Let $\C$ be a stable category. A full subcategory $\L\subseteq \C$ is said to be \emph{thick} if it is a full stable subcategory closed under finite colimits. In particular, it is closed under extensions and desuspensions. We say $\L$ is a \emph{stable localizing subcategory} if it is thick and closed under filtered colimits. 
\end{definition}

Stable localizing subcategories are uniquely determined by localization functors on $\C$, hence their name. This is a standard fact about localizations, but we include a sketch of the proof for convenience. 

\begin{lemma}
    A full subcategory $\L$ of a stable category $\C$, is a stable localizing subcategory if and only if there is a stable category $\D$, and an exact localization $L\:\C\to\D$, such that $\L$ is the kernel of $L$. 
\end{lemma}
\begin{proof}
    Let $\L$ be a localizing subcategory of $\C$. The right-orthogonal complement 
    \[\L^\perp = \{C\in \C \mid \Hom(X, C)\simeq 0, \forall X\in \C\}\]
    is closed under limits in $\C$, hence the fully faithful inclusion $\L^\perp \hookrightarrow \C$ has a left adjoint $L$. This is an exact localization of stable $\infty$-categories, and the kernel is precisely $\L$. Now, given an exact localization $L\: \C\to \D$ such that $\L = \Ker L$, then $\L$ is a stable category by the exactness of $L$, which is in addition closed under colimits as $L$ preserves these by being a left adjoint. 
\end{proof}


The definition of a localizing subcategory of a prestable category is very similar in nature to its stable brethren, but there is a slight variation. 

\begin{definition}
    Let $\C\geqz$ be a Grothendieck prestable category and $C$ an object in $\C$. Another object $C'\in \C$ is said to be a sub-object of $C$ if there is a map $f\: C'\to C$ with $ \Cofib(f)\in \C^\heart$. 
\end{definition}

\begin{remark}
    For Grothendieck prestable categories, this is equivalent to the assertion that $C'$ is a $(-1)$-truncated object in $\C_{/C}$ via the map $f$, which is the more standard definition of being a sub-object --- see \cite[C.2.3.4]{lurie_SAG}
\end{remark}

\begin{definition}[{\cite[C.2.3.3]{lurie_SAG}}]
    Let $\C\geqz$ be a Grothendieck prestable category. A full subcategory $\L\geqz\subseteq \C\geqz$ is a prestable localizing subcategory if it is accessible and closed under coproducts, cofiber sequences and sub-objects. 
\end{definition}

\begin{remark}
    \label{rm:prestable-localizing-is-Grothendieck}
    Any prestable localizing subcategory $\L\geqz$ of a Grothendieck prestable category $\C\geqz$ is by \cite[C.5.2.1]{lurie_SAG} itself a Grothendieck prestable category. This means, in particular, that $\L\geqz$ is the connected part of a colimit-compatible $t$-structure on a stable category, hence using the notation $\L\geqz$ is not abusive. 
\end{remark}

\begin{remark}
    \label{rm:prestable-localizing-in-stable-then-stable-localizing}
    Recall from \cref{rm:stable-is-prestable} that any stable category $\C$ can be treated as a prestable category. By \cite[C.2.3.6]{lurie_SAG} a full subcategory $\L$ of $\C$ is a prestable localizing subcategory if and only if it is a stable localizing subcategory. 
\end{remark}

As in the stable situation we have a description of prestable localizing subcategories via localization functors. 

\begin{proposition}[{\cite[C.2.3.8]{lurie_SAG}}]
    \label{prop:Lurie-prestable-localizing-left-exact-functor}
    A full subcategory $\L\geqz\subseteq \C\geqz$ of a Grothendieck prestable category is localizing if and only if there is a Grothendieck prestable category $\D\geqz$, and left exact localization $L\:\C\geqz\to\D\geqz$, such that $\L\geqz$ is the kernel of $L$. 
\end{proposition}

As prestable localizing subcategories are again prestable, we know that there is some stable category with a $t$-structure presenting it as its connected component. The prestable localizing subcategories hence naturally encodes a sort of induced $t$-structure. This does not happen automatically for stable categories, hence we need to make some additional requirements in order to successfully move between the prestable and stable situation. 

\begin{definition}
    Let $\C$ be a $t$-stable category. A full subcategory $\L\subseteq \C$ is said to be a \emph{$t$-stable localizing subcategory} if it is localizing, and for any $X\in \L$ we have $\tau\geqz X\in \L$ and $\tau\leqz X\in \L$. 
\end{definition}

\begin{remark}
    We hope that using both the names $t$-stable categories and $t$-stable localizing subcategories does not cause confusion. We decided to use this terminology, as a $t$-stable localizing subcategory is itself a $t$-stable category, as we will see in \cref{lm:localizing-inherits-completeness-and-colimits}. 
\end{remark}

\begin{remark}
    Let $\L$ be a $t$-stable localizing subcategory of $\C$. As localizing subcategories are stable under (de)suspension, this means that also all $\tau\geqn X$ and $\tau\leqn$ lie in $\L$ for all $n$. In particular, the homotopy groups $\pi_n^\heart X$ lie in $\L$ for all $n$. 
\end{remark}

\begin{remark}
    \label{rm:t-stable-truncation-homotopy-the-same-functors}
    This definition is motivated by \cite[1.3.19]{beilinson-bernstein-deligne_1983}, where the authors prove that such a full subcategory inherits a $t$-structure given by 
    \[(\L\geqz, \L\leqz) = (\C\geqz\cap \L, \C\leqz\cap \L)\]
    with heart $\C^\heart\cap \L$. In other words, a $t$-stable localizing subcategory has a ``sub $t$-structure'', such that the inclusion is $t$-exact. In particular, the truncation functors $\tau\geqn$ and $\tau\leqn$ are the same as those in $\C$, hence also the homotopy group functors $\pi_n^\heart$ are the same in $\C$ and $\L$. 
\end{remark}



We will from now on assume that a $t$-stable localizing subcategory is equipped with the above $t$-structure. 

\begin{proposition}
    \label{prop:induced-t-structure-on-stable-localizing}
    Let $\C$ be a stable category with a right complete $t$-structure and let $\L\subseteq \C$ be a localizing subcategory. If $\L$ is $t$-stable, then the induced $t$-structure on $\L$ is right complete.  
\end{proposition}
\begin{proof}
    This follows immediately from the fact that the truncation functors are the same as in $\C$, and that colimits in $\L$ are the same as those in $\C$. 
\end{proof}

The last thing to introduce in this section are the abelian analogs of the above definitions. 

\begin{definition}
    A full subcategory $\T$ of a Grothendieck abelian category $\A$ is called a \emph{weak Serre subcategory}, if for any exact sequence 
    \[A_1 \to A_2 \to A_3 \to A_4 \to A_5\]
    in $\A$ such that $A_1, A_2, A_4, A_5$ are all in $\T$, then also $A_3 \in \T$. It is a \emph{abelian weak localizing subcategory} it it is a weak Serre subcategory closed under arbitrary coproducts. 
\end{definition}

\begin{remark}
    A full subcategory is a weak Serre subcategory if it is closed under kernels, cokernels and extensions. In particular it is an abelian subcategory, and the fully faithful inclusion $\T\hookrightarrow \A$ is exact. 
\end{remark}

\begin{definition}
    A full subcategory $\T$ of a Grothendieck abelian category $\A$ is called a \emph{Serre subcategory} if for any short exact sequence 
    \[0\to A\to B\to C\to 0\]
    in $\A$, we have $B\in \T$ if and only if $A, C \in \T$. It is an \emph{abelian localizing subcategory} if it is a Serre subcategory closed under arbitrary coproducts. 
\end{definition}

\begin{remark}
    A full subcategory is a Serre subcategory if it is closed under sub-objects, quotients and extensions. This means that all Serre subcategories are weak Serre subcategories, and that all abelian localizing subcategories are abelian weak localizing subcategories. In particular they are all abelian subcategories with exact inclusions into $\A$. 
\end{remark}

\begin{remark}
    Weak Serre subcategories seem to also be called \emph{thick} or \emph{wide} subcategories in the homological algebra literature. But, to make the connection with abelian localizing subcategories clearer we chose to use this terminology. 
\end{remark}

As one perhaps should expect at this point, Abelian localizing subcategories are also determined by localization functors --- as above, so below. 

\begin{proposition}[{\cite[C.5.1.1, C.5.1.6]{lurie_SAG}}]
    \label{prop:abelian-localizing-iff-kernel-of-localization}
    A full subcategory $\T$ of a Grothendieck abelian category $\A$ is an abelian localizing subcategory if and only if there is an exact localization $L\:\A\to \B$, where $\B$ is a Grothendieck abelian category, such that $\T$ is the kernel of $L$. 
\end{proposition}



\subsection{Stable and prestable comparisons}

The first thing we need is to be able to recognize stable localizing subcategories by their connected part, as we did for stable categories in \cref{lm:right-complete-then-equiv-to-sp}. 

\begin{corollary}
    \label{cor:t-stable-implies-equiv-to-sp}
    Let $\C$ be a stable category with a right complete $t$-structure and $\L$ a $t$-stable localizing subcategory. In this situation there is an equivalence $\L\simeq \Sp(\L\geqz)$. 
\end{corollary}
\begin{proof}
    This follows directly from \cref{prop:induced-t-structure-on-stable-localizing} and \cref{lm:right-complete-then-equiv-to-sp}.
\end{proof}

Using this we can increase the strength of \cref{prop:induced-t-structure-on-stable-localizing} by also incorporating compatibility with filtered colimits. Recall that we use the name $t$-stable category for a stable category with a right complete $t$-structure compatible with filtered colimits. 

\begin{lemma}
    \label{lm:localizing-inherits-completeness-and-colimits}
    Let $\C$ be a $t$-stable category and $\L$ a localizing subcategory. If $\L$ is $t$-stable, then $\L$ is itself a $t$-stable category. 
\end{lemma}
\begin{proof}
    By \cref{prop:induced-t-structure-on-stable-localizing} we know that the induced $t$-structure on $\L$ is right complete. By \cite[C.5.2.1(1)]{lurie_SAG} $\L\geqz$ is Grothendieck prestable, hence the $t$-structure on its stabilization $\Sp(\L\geqz)$ is compatible with filtered colimits by definition, see \cite[C.1.4.1]{lurie_SAG}. This stabilization is by \cref{cor:t-stable-implies-equiv-to-sp} equivalent to $\L$, completing the proof. 
\end{proof}

Recall that any stable localizing subcategory $\L\subseteq \C$ is equivalently determined as the acyclic objects to an exact localization functor $L\:\C\to\D$. We want a similar fact to hold for the $t$-stable ones. The naïve guess could perhaps be that $\L$ is $t$-stable if and only if the localization functor $L$ is $t$-exact. This turns out to be too strong of a condition on the nose, but a  very interesting condition nonetheless. 

\begin{lemma}
    \label{lm:t-exact-then-t-stable-kernel}
    Let $L\:\C\to\D$ be a localization of stable categories with $t$-structures. If $L$ is $t$-exact, then $\Ker(L)$ is a $t$-stable localizing subcategory.  
\end{lemma}
\begin{proof}
    Let $X\in \Ker(L)$. Since $L$ is $t$-exact we have $L(\tau\geqz X)\simeq \tau\geqz L(X)\simeq 0$, hence also $\tau\geqz X$ is in $\Ker(L)$. We have $\tau\leqz X\in \Ker(L)$ by an identical argument.  
\end{proof}

We can then relate this to the prestable situation via the following lemma. 

\begin{lemma}[{\cite[C.2.4.4]{lurie_SAG}}]
    \label{lm:exact-functor-induces-left-exact-functor}
    If $F\:\C\to \D$ is an $t$-exact functor between stable categories with right complete $t$-structures, then the induced functor of Grothendieck prestable categories $F\geqz\:\C\geqz\to\D\geqz$ is left exact. 
\end{lemma}

\begin{remark}
    \label{rm:kernel-of-t-exact-then-prestable-localizing}
    Since prestable localizing subcategories are determined by left exact localization functors, see \cref{prop:Lurie-prestable-localizing-left-exact-functor}, \cref{lm:exact-functor-induces-left-exact-functor} means that if $\L$ is a stable localizing subcategory determined by a $t$-exact localization functor $\C\to \D$, then the connected part $\L\geqz$ is a prestable localizing subcategory of $\C\geqz$. 
\end{remark}

We also want a converse to this statement.

\begin{lemma}
    \label{lm:prestable-localizing-then-kernel-of-t-exact}
    If $\L\geqz$ is a prestable localizing subcategory of a Grothendieck prestable category $\C\geqz$, then its stabilization $\Sp(\L\geqz)$ is the kernel of a $t$-exact localization $L$ on $\Sp(\C\geqz)$.  
\end{lemma}
\begin{proof}
    By \cref{prop:Lurie-prestable-localizing-left-exact-functor} there is a left exact localization $L\geqz \: \C\geqz \to \D\geqz$ such that $\L\geqz$ is the kernel of $L\geqz$. In particular it is a colimit preserving functor. The induced functor $\Sp(\L\geqz)\:\Sp(\C\geqz)\to \Sp(\D\geqz)$ is then left $t$-exact by \cite[C.3.2.1]{lurie_SAG} and right $t$-exact by \cite[C.3.1.1]{lurie_SAG}. 
\end{proof}

\begin{remark}
    In particular, by \cref{lm:t-exact-then-t-stable-kernel} the stabilization $\Sp(\L\geqz)$ is a $t$-stable localizing subcategory. 
\end{remark}

In light of the above results we introduce the following definition. 

\begin{definition}
    A stable localizing subcategory $\L\subseteq \C$ is said to be \emph{$t$-exact} if it is the kernel of a $t$-exact localization. 
\end{definition}

\begin{remark}
    \label{rm:recurring-1}
    As we will have several definitions for different kinds of localizing subcategories, we will have a recurring remark about their dependencies. In this first such remark, we note that there is an implication
    \[t\text{-exact}\implies t\text{-stable}\]
    by \cref{lm:t-exact-then-t-stable-kernel}. 
\end{remark}

We can then conclude this section with the following bijection. 

\begin{corollary}
    \label{cor:t-exact-corresponds-to-prestable-localizing}
    For any $t$-stable category $\C$, there is a bijection between the collection of $t$-exact stable localizing subcategories $\L\subseteq \C$, and prestable localizing subcategories of $\C\geqz$, given by the mutually inverse functors $(-)\geqz$ and $\Sp(-)$. 
\end{corollary}
\begin{proof}
    From \cref{rm:kernel-of-t-exact-then-prestable-localizing} and \cref{lm:prestable-localizing-then-kernel-of-t-exact} we have maps 
    \[\texact\overset{(-)\geqz}\to \prlocalizing\] 
    and 
    \[\prlocalizing\overset{\Sp(-)}\to \texact\] 
    These are mutually inverse functors by \cref{cor:t-stable-implies-equiv-to-sp}, and the fact that any prestable localizing subcategory of a Grothendieck prestable category is itself a Grothendieck prestable category, see \cref{rm:prestable-localizing-is-Grothendieck}. 
\end{proof}


\begin{remark}
    The above corollary gives us a $t$-exact approximation result for $t$-stable localizing subcategories. Suppose we have a $t$-stable localizing subcategory $\L\subseteq \C$. We can choose the smallest prestable localizing subcategory of $\C\geqz$ containing $\L\geqz$, which we denote $\Loc\geqz(\L\geqz)$. Upon stabilization we obtain by \cref{cor:t-exact-corresponds-to-prestable-localizing} a stable localizing subcategory $\L^t$
    that is the kernel of a $t$-exact functor. As $\Sp(\L\geqz)\simeq \L$, we know that $\L\subseteq \L^t$, making $\L^t$ a $t$-exact approximation of $\L$. It is also the smallest such approximation, and, naturally, $\L$ is $t$-exact if and only if $\L\simeq \L^t$. 
\end{remark}

\section{The correspondences}

The goal of this section is to prove our two main results. We start with the classification of weak localizing subcategories, before proving the non-weak case. The former does not need any of the connections to prestable categories, hence can also be viewed as a self contained argument. The latter, however, relies on Lurie's correspondence between certain prestable localizing subcategories of $\C\geqz$ and localizing subcategories of $\C^\heart$.

\subsection{Classification of weak localizing subcategories}
\label{ssec:classificartion-weak-localizing}

The goal of this section is to prove \cref{thm:A}, and the following lemma is the first step for obtaining the wanted correspondence. 

\begin{lemma}
    \label{lm:t-stable-then-weak-localizing-heart}
    Let $\C$ be a $t$-stable category. If $\L$ is a $t$-stable localizing subcategory, then $\L^\heart$ is a weak localizing subcategory of $\C^\heart$. 
\end{lemma}
\begin{proof}
    As $\L$ is $t$-stable we know that the fully faithful inclusion $\L\to \C$ is $t$-exact. By \cite[2.19]{antieau-gepner-heller_2019} the induced functor $\L^\heart\to\C^\heart$ is exact and fully faithful, and $\L^\heart$ is closed under extensions. In particular, $\L^\heart$ is an abelian subcategory closed under extensions, so it remains only to show that $\L^\heart$ is closed under coproducts.

    As $\L^\heart \subseteq \L$ we can include a coproduct of objects in $\L^\heart$ into $\L$. The inclusion and $\pi_n^\heart$ preserves coproducts for all $n$. Hence, as $\L$ is localizing it is closed under coproducts, implying that also $\L^\heart$ is. 
\end{proof}

This means that the heart construction $\C \longmapsto \C^\heart$ determines a map
\[\tcompatible\overset{(-)^\heart}\to \weaklocalizing\] 
for any $t$-stable category $\C$. 

This map is in general not injective, meaning we have to restrict our domain. As described in the introduction, we will use the localizing subcategories where objects can be identified by their heart-valued homotopy groups. The precise definition is as follows. 

\begin{definition}
    \label{def:pi-stable-localizing-subcategory}
    Let $\C$ be a stable category with a $t$-structure. A stable localizing subcategory $\L$ is said to be \emph{$\pi$-stable} if $X\in \L$ if and only if $\pi_n^\heart X\in \L^\heart$ for all $n$. 
\end{definition}

\begin{remark}
    The terminology is motivated by, and generalizes, Takahashi's definition of $H$-stable subcategories of the unbounded derived category of a commutative noetherian ring, see \cite[2.11]{takahashi_2009}. These are subcategories of derived categories where one can detect containment by checking on homology. Letting $\C=\Der(R)$ for a Noetherian commutative ring $R$ considered with the natural $t$-structure, then we have $\pi_n^\heart = H_n$, meaning that being $\pi$-stable is equivalent to being $H$-stable. Note, however, that the homological algebra literature often uses cohomological indexing, while we follow Lurie's convention of using the homological one. 
\end{remark}

\begin{remark}
    The above definition is equivalent to Zhang--Cai's generalization of Takahashi's $H$-stable subcategories, see \cite{zhang-cai_2017}. Note that the authors of loc. cit. do not consider the subcategories themselves to have $t$-structures, but rather just includes the image of $\pi_k^\heart$ back into the stable category. 
\end{remark}

\begin{example}
    \label{ex:abelian-torsion}
    Let $R$ be a commutative noetherian ring and $I\subseteq R$ a finitely generated regular ideal. Then the full subcategory of $I$-power torsion modules, $\Mod_R^{I-tors}\subseteq \Mod_R$ is an abelian weak localizing subcategory. It is in particular a Grothendieck abelian category, hence has a derived category $\Der(\Mod_R^{I-tors})$. We can also form the derived $I$-torsion category $\Der(R)^{I-tors}$, which is the localizing subcategory generated by $A/I$. The categories $\Der(R)^{I-tors}$ and $\Der(\Mod_R^{I-tors})$ are both $\pi$-stable localizing subcategories of $\Der(R)$ with heart $\Mod_R^{I-tors}$ --- see \cite{greenlees-may_92} or \cite{barthel-heard-valenzuela_2018} for more details. These categories are equivalent, seemingly implying that having the same heart is enough for the stable categories to be equivalent as well. This also generalizes to other similar situations, see for example \cite[3.15, 3.17]{barthel-heard-valenzuela_2020} or \cite[2.21]{aambo_2024}. Such equivalences were one of the main inspirations for this paper, where the author wanted an easier way of checking similar statements, which led to the main result \cref{thm:A}. 
\end{example}

\begin{proposition}
    \label{prop:pi-stable-then-t-stable}
    Let $\L$ be a localizing subcategory of $\C$. If $\L$ is $\pi$-stable, then $\L$ is $t$-stable. 
\end{proposition}
\begin{proof}
    Let $X\in \L$. We need to show that $\tau\geqz X\in \L$ and $\tau\leqz X\in \L$. The proofs are similar, hence we only cover the former. 
    
    We have $\pi_n^\heart \tau\geqz X \simeq \pi_n^\heart X$ for all $n\geq 0$ and $\pi_n^\heart \tau\geqz X\simeq 0$ for all $n<0$. This means that $\pi_n^\heart \tau\geqz X \in \L^\heart$ for all $n$, which implies $\tau\geqz X \in \L$ by the assumption that $\L$ was $\pi$-stable. 
\end{proof}

\begin{remark}
    \label{rm:recurring-2}
    In light of \cref{prop:pi-stable-then-t-stable} we can continue our recurring remark (see \cref{rm:recurring-1}) about the dependencies of the different definitions. We now have implications
    \begin{center}
        \begin{tikzcd}
            & \pi\text{-stable} \arrow[d, Rightarrow] \\
            t\text{-exact} \arrow[r, Rightarrow] & t\text{-stable}                    
        \end{tikzcd}
    \end{center}
    for any localizing subcategory $\L$ of a $t$-stable category $\C$. 
\end{remark}

\begin{remark}
    \label{rm:pi-stable-implies-t-stable}
    In particular, if $\L$ is a $\pi$-stable localizing subcategory then \cref{lm:localizing-inherits-completeness-and-colimits} implies that $\L$ is itself a $t$-stable category. This is rather convenient, as it allows us to treat nested pairs of $\pi$-stable localizing subcategories $\L_2 \subseteq \L_1 \subseteq \C$ either as both being subcategories of $\C$, or as $\L_2$ being a $\pi$-stable localizing subcategory of $\L_1$. 
\end{remark}

\cref{prop:pi-stable-then-t-stable} implies that the heart construction $\L\longmapsto \L^\heart$ gives a map 
\[\pistable\overset{(-)^\heart}\to \weaklocalizing\] 
as the heart of any $t$-stable localizing subcategory $\L\subseteq \C$ is an abelian weak localizing subcategory $\L^\heart \subseteq \C^\heart$ by \cref{lm:t-stable-then-weak-localizing-heart}. The claim of \cref{thm:A} is that this map is a bijection. 

It turns out that the $\pi$-stable localizing subcategories are the largest localizing subcategories with a given heart. This is the stable analog of \cite[C.5.2.5]{lurie_SAG} for prestable categories. 

\begin{lemma}
    \label{lm:pi-stable-are-the-biggest}
    Let $\C$ be a $t$-stable category. Given two $t$-stable localizing subcategories $\L_0$ and $\L_1$, where $\L_1$ is $\pi$-stable, then $\L_0 \subseteq \L_1$ if and only if $\L_0^\heart \subseteq \L_1^\heart$. 
\end{lemma}
\begin{proof}
    First, notice that as both categories are $t$-stable the truncation functors and the homotopy groups functors $\pi_k^\heart$ are the same, see \cref{rm:t-stable-truncation-homotopy-the-same-functors}. 
    
    Assume $\L_0^\heart \subseteq \L_1^\heart$ and $X\in \L_0$. Then $\pi_k^\heart X\in \L_0^\heart \subseteq \L_1^\heart$ for all $k$. This implies that $X\in \L_1$ by the assumption that it is $\pi$-stable. 

    For the converse, assume $\L_0\subseteq \L_1$. As the truncation functors are the same in $\L_0$ and $\L_1$ we have that $\L_0$ is a $t$-stable localizing subcategory of the $t$-stable category $\L_1$, see \cref{rm:pi-stable-implies-t-stable}. In particular, $\L_0^\heart = \L_1^\heart \cap \L_0$, hence we have $\L_0^\heart \subseteq \L_1^\heart$. 
\end{proof}

This immediately implies the injectivity of our proposed one-to-one correspondence. 

\begin{corollary}
    \label{cor:pi-stable-heart-injective}
    For any $t$-stable category $\C$, the map
    \[\pistable\overset{(-)^\heart}\to \weaklocalizing\] 
    is injective. 
\end{corollary}
\begin{proof}
    Let $\L_0$ and $\L_1$ be $\pi$-stable localizing subcategories such that $\L_0^\heart \simeq \L_1^\heart$ as subcategories of $\C^\heart$. In particular, they are contained in each other, hence \cref{lm:pi-stable-are-the-biggest} implies that $\L_0 \subseteq \L_1$ and $\L_1 \subseteq \L_0$ as they are both $\pi$-stable. 
\end{proof}

It remains to show that the map is also surjective. 

\begin{theorem}[{\cref{thm:A}}]
    \label{thm:premain}
    Let $\C$ be a $t$-stable category. In this situation, the map 
    \[\pistable\overset{(-)^\heart}\to \weaklocalizing\] 
    is a bijection. 
\end{theorem}
\begin{proof}
    We know by \cref{cor:pi-stable-heart-injective} that the map is injective, hence it remains to prove surjectivity. To do this we follow the proof of \cite[C.5.2.7]{lurie_SAG}, adapted to the stable setting. 
    
    Let $\A$ be a weak localizing subcategory of $\C^\heart$. Define $\L\subseteq \C$ to be the full subcategory spanned by objects $X$ such that $\pi_n^\heart X\in \A$. We prove that it is a stable localizing subcategory --- it will obviously be $\pi$-stable by definition. In particular we prove that it is closed under cofiber sequences, desuspension and colimits. 

    Let $A\rightarrow B\rightarrow C$ be a cofiber sequence in $\C$. We need to show that if two of the objects $A, B, C$ is in $\L$, then also the last one is. The long exact sequence of heart-valued homotopy groups has the form 
    \[\cdots\rightarrow \pi^{\heart}_{n-1}B \rightarrow \pi^{\heart}_{n-1}C \rightarrow \pi^{\heart}_{n}A \rightarrow \pi^{\heart}_{n}B \rightarrow \pi^{\heart}_{n}C \rightarrow \pi^{\heart}_{n+1}A\rightarrow \pi^{\heart}_{n+1}B \rightarrow \cdots \]
    Assuming that $A, B$ are in $\L$ we get by the definition of $\L$ that the four objects $\pi^{\heart}_{n}A, \pi^{\heart}_{n}B, \pi^{\heart}_{n+1}A, \pi^{\heart}_{n+1}B$ are in $\A$. Hence, as $\A$ is a weak Serre subcategory we have $\pi^\heart_n C\in \A$. This works for all $n$, hence we must have $C\in \L$ as well. The proof is identical in the case that $A, C$ or $B, C$ are in $\L$. 
    
    The full subcategory $\L$ is also closed under desuspension, as we have $\pi_n^\heart (\Omega X)\simeq \pi_{n+1}^\heart(X)$ by the long exact sequence in heart-valued homotopy groups. Hence $\L$ is a full stable subcategory of $\C$. In particular this means it is closed under finite colimits. Now, as $\pi_n^\heart$ preserves coproducts, and $\A$ is closed under coproducts, we also get that $\L$ is closed under coproducts. This implies that $\L$ is closed under colimits, which finishes the proof. 
\end{proof}

\begin{remark}
    It is somewhat unfortunate that the terminology does not align perfectly in these two situations --- meaning that we had to add a prefix ``weak'' for the abelian case. As both are inspired by the existence of localization functors, they are the natural terminology in their respective settings, and we should perhaps not expect everything to always agree perfectly. In \cref{thm:B} we will use the abelian localizing subcategories, and then again be left with a choice of a different prefix for the stable version. 
\end{remark}

\cref{thm:A} recovers, and generalizes, a theorem by Takahashi for commutative noetherian rings. Note that Takahashi does not refer to the abelian subcategories as weak localizing, but as thick subcategories closed under coproducts. 

\begin{corollary}[{\cite{takahashi_2009}}]
    \label{cor:takahashi-weak-localizing}
    If $R$ is a commutative noetherian ring, then there is a bijection between the set of $H$-stable localizing subcategories of $\Der(R)$ and the set of weak localizing subcategories in $\Mod_R$. 
\end{corollary}

A theorem of Krause --- see \cite[3.1]{krause_2008} --- shows that these two collections are also in bijection with certain subsets of $\Spec R$, which Krause calls the \emph{coherent subsets}. In light of \cref{thm:A} we can generalize Takahashi's result to a noetherian scheme $X$, and we conjecture that these are also in bijection with the coherent subsets of $X$ --- generalizing the result by Krause.  

\begin{corollary}
    \label{cor:noetherian-scheme-weak-localizing}
    If $X$ noetherian scheme, then there is a bijection between the set of stable localizing subcategories of $\Der_{qc}(X)$ closed under homology, and the set of weak localizing subcategories in $\QCoh(X)$. 
\end{corollary}

\begin{conjecture}
    \label{conj:coherent-noetherian-scheme}
    For a noetherian scheme $X$, there is a bijection between the collection of coherent subsets of $X$ and weak localizing subcategories of $\QCoh(X)$. 
\end{conjecture}

\begin{remark}
    A hint towards the truth of this conjecture comes from a theorem by Gabriel (\cite[VI.2.4(b)]{gabriel_1962}), where he shows that the above proposed bijection restricts to a bijection between specialization closed subsets of $X$ and localizing subcategories of $\QCoh(X)$. 
\end{remark}

Now, we want to mention that we also obtain a classification of weak Serre subcategories of $\C$. This is done by recognizing that the proofs of \cref{lm:t-stable-then-weak-localizing-heart}, \cref{cor:pi-stable-heart-injective} and \cref{thm:A} also holds without the assumption about coproducts. The proofs treats coproducts as a separate part, hence just omitting it from the proofs gives the following result. 

\begin{proposition}
    \label{prop:classification-weak-serre}
    Let $\C$ be a $t$-stable category. In this situation, the map 
    \[\pistablethick\overset{(-)^\heart}\to \weakserre\] 
    is a bijection. 
\end{proposition}

This recovers the following classification of weak Serre subcategories in the case where the $t$-structure on $\C$ is bounded, due to Zhang--Cai, see \cite{zhang-cai_2017}. 

\begin{corollary}
    \label{cor:classification-weak-serre-bounded}
    Let $\C$ be a triangulated category with a bounded $t$-structure. In this situation there is a bijection between $\pi$-stable subcategories of $\C$ and weak Serre subcategories of $\C^\heart$.  
\end{corollary}

We can summarize the contents of this section with half of the diagram from the introduction. 

\begin{center}
    \begin{tikzcd}
        \tcompatible 
        \arrow[rrd, "(-)^\heart"]
        &&\\
        \pistable 
        \arrow[rr, "\simeq"] 
        \arrow[u, hook, "\subseteq", swap] 
        && 
        \weaklocalizing                          
        \end{tikzcd}
\end{center}

\subsection*{Digression on Grothendieck homology theories}

There is a slight generalization of the surjectivity result above, which we decided to include here for future reference. The generalization comes from realizing that there are other functors that have similar properties to the heart valued homotopy group functor $\pi_*^\heart \: \C\to \C^\heart$. 

Let $\C$ be a presentable stable $\infty$-category and $\A$ be a graded Grothendieck abelian category --- meaning it comes equipped with an autoequivalence $[1]\: \A\to \A$, which we think of as a grading shift functor. 

\begin{definition}
    A functor $H\: \C\to \A$ is called a \emph{Grothendieck homology theory} if it satisfies the following properties:
    \begin{enumerate}
        \item It is additive.
        \item It sends cofiber sequences $X\rightarrow Y \rightarrow Z$ to exact sequences $HX\rightarrow HY\rightarrow HZ$.
        \item It is a graded functor, i.e. $H(\Sigma X) \cong (HX)[1]$.
        \item It preserves coproducts. 
    \end{enumerate}
\end{definition}

\begin{remark}
    The first two criteria defines $H$ to be what is usually called a homological functor. Adding the third criteria makes $H$ a homology theory, and the last is what makes it Grothendieck. 
\end{remark}

The main example of these come from the category of spectra, $\Sp$, where the associated homology theory to any spectrum is a Grothendieck homology theory.

\begin{example}
    Let $\C=\Sp$ and $R$ be a graded commutative ring. The Eilenberg--MacLane spectrum $HR$ is a commutative ring spectrum, and the associated homology theory $HR_* := [\S, HR\otimes (-)]_*\: \Sp \to \Mod_R$ is a Grothendieck homology theory. This homology theory is equivalent to singular homology with $R$ coefficients. 
\end{example}

The above example holds more generally as well. 

\begin{example}
    If $\C$ is monoidal and the unit $\1$ is compact, then for any $H\in \C$ the associated functor
    \begin{align*}
        H_* \: \C &\to \Ab\gr \\
        X &\longmapsto [\1, H\otimes X]_*
    \end{align*}
    is a Grothendieck homology theory. 
\end{example}

\begin{proposition}
    Let $H\: \C\to \A$ be a Grothendieck homology theory and $\T$ a weak localizing subcategory of $\A$. In this situation, the full subcategory $\L\subseteq \C$ consisting of objects $X$ such that $HX\in \T$, is a localizing subcategory of $\C$. 
\end{proposition}
\begin{proof}
    This holds by using the same surjectivity argument from \cref{thm:premain}, just exchanging $\pi_n^\heart(-)$ with $H(-)[n]$. 
\end{proof}

This gives a commutative diagram
\begin{center}
    \begin{tikzcd}
        \C \arrow[r, "H"]           & \A           \\
        \L \arrow[r, "H"] \arrow[u] & \T \arrow[u]
    \end{tikzcd}    
\end{center}
where both of the fully faithful vertical functors have right adjoints. Note that the adjoint diagram might not commute. 

\begin{remark}
    In addition to being a localizing subcategory, we have by definition that we can check containment of $\L$ on the associated Grothendieck abelian category $\T$. This means that $\L$ also has a certain $\pi$-stability property, which one might call being $H$-stable, generalizing both \cref{def:pi-stable-localizing-subcategory} and Takahashi's notion of $H$-stability. 
\end{remark}

\subsection{Classification of localizing subcategories}
\label{ssec:classificartion-localizing}

The goal of this section is to prove \cref{thm:B}, and that it interacts well with both Lurie's classification via prestable categories, and \cref{thm:A}. As in \cref{ssec:classificartion-weak-localizing} we start by proving that the wanted map of sets exists.

\begin{lemma}
    \label{lm:pi-exact-then-abelian-localizing-heart}
    Let $\C$ be a $t$-stable category. If $\L$ is a $t$-exact localizing subcategory, then $\L^\heart$ is an abelian localizing subcategory of $\C^\heart$. 
\end{lemma}
\begin{proof}
    The $t$-exact localization $L\:\C\to \D$ and its right adjoint $i$ induces an adjunction 
    \begin{center}
        \begin{tikzcd}
            \C^\heart \arrow[r, "L^\heart", yshift=2pt] & \D^\heart \arrow[l, yshift=-2pt]
        \end{tikzcd}
    \end{center}
    on the corresponding hearts. As $L$ was $t$-exact, the functor $L^\heart$ is exact. In particular, the heart $\L^\heart$ is the kernel of $L^\heart$, which by \cref{prop:abelian-localizing-iff-kernel-of-localization} means that $\L^\heart$ is an abelian localizing subcategory of $\C^\heart$. 
\end{proof}

This means that we have a map 
\[\texact\overset{(-)^\heart}\to \ablocalizing\]

Just as for the non-$t$-exact case, this map is not injective in general, meaning we have to restrict to a type of subcategory with more structure. 




\begin{definition}
    \label{def:pi-exact-localizing-subcategory}
    A localizing subcategory $\L$ of a $t$-stable category $\C$ is said to be a $\pi$-exact localizing subcategory if 
    \begin{enumerate}
        \item it is $\pi$-stable, and
        \item it is the kernel of a $t$-exact localization. 
    \end{enumerate}
\end{definition}

\begin{remark}
    \label{rm:recurring-3}
    We continue our recurring remark about the dependencies of the different kinds of localizing subcategories introduced in the paper, see \cref{rm:recurring-1} and \cref{rm:recurring-2}. We now have implications
    \begin{center}
        \begin{tikzcd}
            \pi\text{-exact} 
            \arrow[d, Rightarrow] 
            \arrow[r, Rightarrow] 
            & 
            \pi\text{-stable} 
            \arrow[d, Rightarrow] 
            \\
            t\text{-exact} 
            \arrow[r, Rightarrow]                         
            & 
            t\text{-stable}                    
        \end{tikzcd}
    \end{center}
    for any localizing subcategory $\L$ of a $t$-stable category $\C$. 
\end{remark}

\begin{remark}
    The above remark also shows how the classification results are related. By \cref{thm:A} we know that $\pi$-stable corresponds to abelian weak localizing subcategories, and by \cref{cor:t-exact-corresponds-to-prestable-localizing} we know that $t$-exact corresponds to prestable localizing subcategories. By Lurie's classification, see \cref{thm:Lurie-correspondence}, we should expect the combination of the two to yield a correspondence between $\pi$-exact localizing subcategories and abelian localizing subcategories. 
\end{remark}

As $\pi$-exact localizing subcategories are by definition $t$-exact, we immediately get that the map $(-)^\heart$ restricts to a map 
\[\piexact\overset{(-)^\heart}\to \ablocalizing\]
The claim of \cref{thm:B} is that this map is a bijection.  

The $\pi$-exact localizing subcategories are the stable analogs of Lurie's notion of separating prestable localizing subcategories, defined as follows.

\begin{definition}
    Let $\C\geqz$ be a Grothendieck prestable category. A prestable localizing subcategory $\L\geqz \subseteq \C\geqz$ is \emph{separating} if for every $X\in \C\geqz$ such that $\pi_n^\heart X\in \L^\heart$ for all $n$, then $X\in \L\geqz$. 
\end{definition}

What we mean by saying that these are the stable analogs, is that the bijection
\[\texact \overset{(-)\geqz}\to \prlocalizing\]
from \cref{cor:t-exact-corresponds-to-prestable-localizing} restricts to a bijection between $\pi$-exact stable localizing subcategories and separating prestable localizing subcategories. We prove this in two steps. 

\begin{lemma}
    \label{lm:pi-exact-then-separating}
    Let $\C$ be a $t$-stable category. If $\L$ is a $\pi$-exact localizing subcategory of $\C$, then $\L\geqz$ is a separating localizing subcategory of $\C\geqz$. 
\end{lemma}
\begin{proof}
    By \cref{cor:t-exact-corresponds-to-prestable-localizing} we know that $\L\geqz$ is a prestable localizing subcategory of $\C\geqz$, so it remains to check that it is separating. Assume $X\in \C\geqz$ and $\pi_n^\heart X \in \L^\heart$ for all $n\geq 0$. Treating $X$ as an object in $\C$ via the inclusion $\C\geqz \hookrightarrow \C$ we have $\pi_i^\heart X = 0$ for all $i<0$. Hence,by the assumption that $\L$ is $\pi$-stable, we must have $X\in \L$. This means that $X \in \C\geqz\cap \L = \L\geqz$, which finishes the proof. 
\end{proof}

\begin{lemma}
    \label{lm:separating-then-pi-exact}
    If $\L\geqz$ is a separating prestable localizing subcategory of $\C\geqz$, then $\Sp(\L\geqz)$ is a $\pi$-exact localizing subcategory of $\C$. 
\end{lemma}
\begin{proof}
    We know by \cref{cor:t-exact-corresponds-to-prestable-localizing} that $\Sp(\L\geqz)$ is a $t$-exact localizing subcategory of $\C$, so it remains to show that it is $\pi$-stable. 

    For the sake of a contradiction, assume that there is some $X\in \C$ with $\pi_n^\heart X\in \L^\heart$ for all $n$, but $X\not\in \L$. As the corresponding localization functor $L\:\C\to \D$ is $t$-exact we get $L\tau\geqz X \simeq \tau\geqz LX$, which is by assumption non-zero, as $X$ was not in $\L$. This means, however, that there is an object $Y=\tau\geqz X$ in $\C\geqz$ with $\pi_n^\heart Y \in \L^\heart$ but $Y$ not in $\L\geqz$, which contradicts $\L\geqz$ begin separating.  
\end{proof}



We are now ready to prove \cref{thm:B}. As for \cref{thm:A} we prove that the map $(-)^\heart$ is both injective and surjective, starting with the former. 

\begin{lemma}
    \label{lm:pi-exact-are-the-biggest}
    Let $\C$ be a $t$-stable category. Given two $t$-exact localizing subcategories $\L_0$ and $\L_1$, where $\L_1$ is $\pi$-exact, then $\L_0 \subseteq \L_1$ if and only if $\L_0^\heart \subseteq \L_1^\heart$. 
\end{lemma}
\begin{proof}
    This immediately follows from the non-$t$-exact case from \cref{lm:pi-stable-are-the-biggest}, as $\L_1$ is $\pi$-stable and $\L_0$ is $t$-stable. 
\end{proof}

As before, this implies that the wanted map is injective. 

\begin{corollary}
    \label{cor:pi-exact-heart-injective}
    For any $t$-stable category $\C$, the map
    \[\piexact \overset{(-)^\heart}\to \ablocalizing\]
    is injective. 
\end{corollary}

It remains then to show that the map is also surjective. In order to do this we invoke Lurie's correspondence. The author originally wanted to have a proof not relying on the prestable case. But, we currently do not know how to directly lift an abelian subcategory to a kernel of a $t$-exact functor, without passing through the bijection from \cref{cor:t-exact-corresponds-to-prestable-localizing}. There is a more direct approach in certain contexts --- for example if the $t$-structure is bounded, see \cite[2.20]{antieau-gepner-heller_2019}, or the inclusion $\L\subseteq \C$ preserves compacts, see \cite[2.7]{hennion-porta-vezzosi_2016} --- but as far as the author is aware, there is no general way to know when the localization determined by a localizing subcategory $\L$ is $t$-exact. 

\begin{theorem}[{\cite[C.5.2.7]{lurie_SAG}}]
    \label{thm:Lurie-correspondence}
    For any Grothendieck prestable category $\C\geqz$, there is a bijection 
    \[\separating \to \ablocalizing\]
    given by $\L\geqz\longmapsto \L^\heart$. 
\end{theorem}

Using this, together with \cref{lm:separating-then-pi-exact} we finally get our wanted one-to-one correspondence. 

\begin{theorem}[\cref{thm:B}]
    \label{thm:main}
    Let $\C$ be a $t$-stable category. There is a bijective map
    \[\piexact\overset{(-)^\heart}\to \ablocalizing\]
    given by $\L\longmapsto \L^\heart$. 
\end{theorem}
\begin{proof}
    The map is injective by \cref{cor:pi-exact-heart-injective}, so it remains only to show surjectivity. Let $\A\subseteq \C^\heart$ be an abelian localizing subcategory. By \cref{thm:Lurie-correspondence} there is a unique separating prestable localizing subcategory $\L\geqz \subseteq \C\geqz$ such that $\L^\heart \simeq \A$. By \cref{lm:separating-then-pi-exact} the spectrum objects in this category, $\Sp(\L\geqz)$ is a $\pi$-exact stable localizing subcategory of $\C$ with heart $\A$. Hence, the map is also surjective. 
\end{proof}


From this we again obtain some natural corollaries. The first one is a partial converse to \cite[2.13]{takahashi_2009}.

\begin{corollary}
    \label{cor:smashing-pi-exact}
    Let $R$ be a commutative noetherian ring and equip $\Der(R)$ with its natural $t$-structure. In this situation there is a bijection between the collection of smashing localizing subcategories and the collection of $\pi$-exact localizing subcategories in $\Der(R)$. 
\end{corollary}
\begin{proof}
    A theorem of Gabriel, see \cite[VI.2.4(b)]{gabriel_1962}, shows that there is a bijection between the collection of localizing subcategories of $\Mod_R$ and specialization closed subsets of $\Spec R$. Further, Neeman shows in \cite[3.3]{neeman_1992} that there is a bijection between specialization closed subsets of $\Spec R$ and smashing localizing subcategories of $\Der(R)$. The result then follows from these, together with \cref{thm:main}. 
\end{proof}

We can also obtain an extension of \cref{cor:smashing-pi-exact} to noetherian schemes $X$. Recall that we denote the abelian category of quasi-coherent sheaves on $X$ by $\QCoh(X)$, and its associated derived category of quasi-coherent $\mathcal{O}_X$-modules by $\Der_{qc}(X)$. 

\begin{lemma}
    For any noetherian scheme $X$, there are bijections
    \[\DXsmashing\simeq \subsetsX \simeq \QCohlocalizing.\]
\end{lemma}
\begin{proof}
    The latter bijection is again due to Gabriel --- \cite[VI.2.4(b)]{gabriel_1962}. By \cite[4.13]{tarrio-lopez-salorio_2004} the telescope conjecture holds for noetherian schemes. In particular, this means that there is a bijection between subsets of $X$ and localizing $\otimes$-ideals in $\Der_{qc}(X)$, see \cite[8.13]{stevenson_2013}, which restricts to a bijection
    \[\DXsmashing\simeq \subsetsX,\]
    giving the first bijection. 
\end{proof}

Utilizing this, together with \cref{thm:main}, we obtain the following generalization. 

\begin{corollary}
    Let $X$ be a noetherian scheme and equip $\Der_{qc}(X)$ with its natural $t$-structure. In this situation, there is a bijection
    \[\DXsmashing \simeq \DXlocalizing.\]
\end{corollary}

Using \cref{cor:noetherian-scheme-weak-localizing} we then get a partial extension of the two bottom rows in the main result of \cite{takahashi_2009} to the case of noetherian schemes.  

\begin{center}
    \begin{tikzcd}
        \DXstable 
        \arrow[r, lightgray, "\simeq"]
        \arrow[rr, "\simeq", bend left = 11]                 
        & 
        \coherentX 
        \arrow[r, lightgray, "\simeq"]                
        & 
        \QCohweak                       \\
        \DXsmashing 
        \arrow[u, hook, "\subseteq", swap] 
        \arrow[r, "\simeq"] 
        & 
        \subsetsX 
        \arrow[u, hook, lightgray, "\subseteq", swap] 
        \arrow[r, "\simeq"] 
        & 
        \QCohlocalizingthree 
        \arrow[u, hook, "\subseteq", swap]
    \end{tikzcd}
\end{center}

Here the grey color indicates the conjectured relationship from \cref{conj:coherent-noetherian-scheme}. 

We can also use the proof of the telescope conjecture for certain algebraic stacks, due to Hall--Rydh (\cite{hall-rydh_2017}), to extend the above corollary even further. We leave the details of this to the interested reader. 

By work of Kanda we can almost extend this to the locally noetherian setting. In particular, for $X$ a locally noetherian scheme, Kanda proves in \cite[1.4]{kanda_2015} that there is a bijection between localizing subcategories of $\QCoh(X)$ and specialization closed subsets of $X$. However, as the telescope conjecture is --- to the best of our knowledge --- currently unresolved for locally noetherian schemes, we do not get a bijection to smashing localizing subcategories. The best we can obtain is then the following corollary. 

\begin{corollary}
    For $X$ a locally noetherian scheme, there are bijections
    \[\DXlocalizing\simeq \subsetsX \simeq \QCohlocalizing.\]
\end{corollary}

\begin{remark}
    It would be very interesting to have a more direct proof for the fact that $\pi$-exact localizing subcategories of $\Der(R)$ and $\Der_{qc}(X)$ corresponds to smashing localizations. Having a direct proof would allow for a new proof of the telescope conjecture for commutative noetherian rings and noetherian schemes, and could shed some new light on the currently unsolved telescope conjecture for locally noetherian schemes. 
\end{remark}

\begin{remark}
    We also want to highlight other work of Kanda, where he shows that localizing subcategories of a locally noetherian Grothendieck abelian category $\A$ are classified by the \emph{atom spectrum} of $\A$, see \cite[5.5]{kanda_classifying_2012}. It would be interesting to see if these atomic methods could provide new insight also into the stable $\infty$-category $\C$. 
\end{remark}

To summarize this section, we construct the bottom part of the diagram from the introduction. By \cref{lm:pi-exact-then-separating} the bijection from \cref{thm:B} factors through the bijection of \cref{thm:Lurie-correspondence}. In particular, we get bijections 
\begin{center}
    \begin{tikzcd}
        \piexact \arrow[r, "(-)\geqz", yshift=2pt] & \separating \arrow[r, "(-)\leqz"] \arrow[l, "\Sp(-)", yshift=-2pt] & \ablocalizing
    \end{tikzcd}
\end{center}
such that the composite map from the left to the right is the map $(-)^\heart$ from \cref{thm:B}. This finally gives the wanted diagram. 

\begin{center}
    \begin{tikzcd}
        \piexact 
        \arrow[r, "\simeq"] 
        \arrow[d, hook, "\subseteq"]
        & 
        \separating 
        \arrow[r, "\simeq"]
        \arrow[d, hook, "\subseteq"]
        & 
        \ablocalizing
        \\
        \texact 
        \arrow[r, "(-)\geqz"]
        & 
        \prlocalizing 
        \arrow[ru, "(-)\leqz", swap] 
        &                                
    \end{tikzcd}
\end{center}

\subsection{Comparing stable categories with the same heart}

We round off the paper by proving some easy corollaries of \cref{thm:A} and \cref{thm:B} for stable categories with $t$-structures with the same heart. The first immediate corollary is the following. 

\begin{corollary}
    Let $\A$ be any Grothendieck abelian category. For any two $t$-stable categories $\C$ and $\D$ with $\C^\heart \simeq \A\simeq \D^\heart$ there are one-to-one correspondences
    \[\pistable\to\pistableD\]
    and 
    \[\piexact\to \piexactD.\]
\end{corollary}

The above correspondence might not be induced by a functor between $\C$ and $\D$, but is just an abstract isomorphism. However, in the case when there is a functor, the $\pi$-stable localizing subcategories are also functorially related. We can set this up as follows. 

\begin{lemma}
    \label{lm:restricted-functor-on-pi-localizing}
    Let $\C$ and $\D$ be $t$-stable categories with $\A^\heart\subseteq \C^\heart$ and $\T^\heart\subseteq \D^\heart$ abelian weak localizing subcategories of the respective hearts. 
    If there is a $t$-exact functor $F\:\C\to \D$ such that the functor on hearts $F^\heart\: \C^\heart\to \D^\heart$ restricts to a functor 
    \[F^\heart_{\vert \A^\heart}\: \A^\heart \to \T^\heart,\]
    then the functor $F$ restricts to the unique $\pi$-stable localizing subcategories $F_{\vert \A}\:\A\to \T$. 
\end{lemma}
\begin{proof}
    As $F$ is $t$-exact we have $F(\pi_{\C,n}^\heart X)\simeq \pi_{\D, n}^\heart F(X)$. By assumption we know that $F(\pi_{\C,n}^{\heart}X) \simeq F^\heart(\pi_{\C,n}^\heart X) \in \T^\heart$, hence any $Y$ in the image of $F$ has $\pi_{\D,n}^\heart Y \in \T^\heart$ for any $n$. Since $\T$ is $\pi$-stable this implies that $Y\in \T$, proving the claim.
\end{proof}

Let $F\:\C\to\D$ be a $t$-exact functor of $t$-stable categories such that the induced functor $F^\heart \: \C^\heart\overset{\simeq}\to \D^\heart$ is an equivalence. Assume further that $\A$ is an abelian weak localizing subcategory of $\C^\heart$, and that $F^\heart$ restricts to a functor $F^\heart_{\vert \A}\:\A\to\A$. By \cref{lm:restricted-functor-on-pi-localizing} we get restricted functors $F_{\vert \A_\C}\: \A_\C \to \A_\D$, where $\A_\C$ and $\A_\D$ respectively denote the unique $\pi$-stable localizing subcategories of $\C$ and $\D$ obtained via \cref{thm:A}.

\begin{corollary}
    If $F$ is an equivalence, then every such restricted functor $F_{\vert \A_\C}$ is an equivalence. 
\end{corollary}

One interesting feature of the $\infty$-categorical framework is the existence of realization functors in reasonable generalities. If $\C$ is a $t$-stable category, then a realization functor for $\C$ is a functor $R\:\Der(\C^{\heart})\to \C$, extending the inclusion of the heart. In particular, $R$ restricts to the identity on $\Der(\C^\heart)^\heart\simeq \C^\heart$. These realization functors are rarely equivalences, even rarely full or faithful, but we can still apply \cref{lm:restricted-functor-on-pi-localizing} to functorially relate the $\pi$-stable localizing subcategories. Note that as $R$ restricts to the identity in hearts, we dont even need to assume or prove that the functor $R$ is $t$-exact, as the proof of \cref{lm:restricted-functor-on-pi-localizing} goes through regardless.  

The following argument is due to Maxime Ramzi.

\begin{lemma}
    Let $\C$ be a $t$-stable category and $\Der(\C^\heart)$ the derived category of its heart. In this situation there is a realization functor $R\:\Der(\C^{\heart})\to \C$ extending the inclusion $\C^\heart\hookrightarrow \C$. 
\end{lemma}
\begin{proof}
     The inclusion of the heart extends to a functor $\Fun(\Delta\op, \C^\heart)\to \C$ via geometric realization, which preserves weak equivalences by \cite[1.2.4.4, 1.2.4.5]{Lurie_HA}. Via the Dold--Kan correspondence this gives a essentially unique colimit preserving functor $\Der(\C^\heart)\geqz\to \C$, which extends uniquely to a functor $\Der(\C^\heart)\to \C$ by \cite[1.4.4.5]{Lurie_HA}, as $\C$ is stable. This functor preserves both colimits and the heart $\C^\heart$. 
\end{proof}

We can then functorially relate the $\pi$-stable localizing subcategories of $\Der(\C^\heart)$ and $\C$ via the realization functor. 

\begin{corollary}
    Let $\C$ be a $t$-stable category and $R\: \Der(\C^\heart)\to \C$ the realization functor. For any weak localizing subcategory $\A\subseteq \C^\heart$, the functor $R$ restricts to a functor
    \[R\: \A_{\Der(\C^\heart)}\to \A_\C,\]
    where the former category denotes the unique $\pi$-stable lift of $\A$ to $\Der(\C^\heart)$, and the latter the unique $\pi$-stable lift of $\A$ to $\C$. 
\end{corollary}
\begin{proof}
    This follows immediately from \cref{lm:restricted-functor-on-pi-localizing}, the $\pi$-stability of $\A_\C$ and the fact that the identity restricts to the identity functor $\A\simeq\A_{\Der(\C^\heart)}^\heart \to \A_\C^\heart\simeq \A$. 
\end{proof}

\begin{remark}
    By \cref{prop:pi-stable-then-t-stable} the $\pi$-stable localizing subcategory $\A_\C$ is also $t$-stable, with heart $\A$. Hence, there is also a realization functor $R'\: \Der(\A)\to \A_\C$, and a natural question to ask is wether this coincides with the above restricted functor $R\: \A_{\Der(\C^\heart)}\to \A_\C$. There is an inclusion $\Der(\A)\subseteq \A_{\Der(\C^\heart)}$, as the latter is a $\pi$-stable localizing subcategory of $\Der(\C^\heart)$, but we do not know if this is always an equivalence. In particular, we don't know whether $\D(\A)$, treated as a subcategory of $\Der(\C^\heart)$, is always a $\pi$-stable localizing subcategory. 
\end{remark}




\printbibliography{}

@article{kanda_classifying_2012,
	title = {Classifying {Serre} subcategories via atom spectrum},
	volume = {231},
	issn = {0001-8708},
	url = {https://www.sciencedirect.com/science/article/pii/S0001870812002599},
	doi = {10.1016/j.aim.2012.07.009},
	number = {3},
	journal = {Advances in Mathematics},
	author = {Kanda, Ryo},
	year = {2012},
	pages = {1572--1588},
}

@article{neeman_1992,
	title = {The chromatic tower for $D(R)$},
	volume = {31},
	issn = {0040-9383},
	url = {https://www.sciencedirect.com/science/article/pii/004093839290047L},
	doi = {10.1016/0040-9383(92)90047-L},
	number = {3},
	journal = {Topology},
	author = {Neeman, Amnon and Bökstedt, Marcel},
	year = {1992},
	pages = {519--532},
}

@article{krause_2008,
	title = {Thick subcategories of modules over commutative noetherian rings (with an appendix by {Srikanth} {Iyengar})},
	volume = {340},
	issn = {1432-1807},
	url = {https://doi.org/10.1007/s00208-007-0166-3},
	doi = {10.1007/s00208-007-0166-3},
	number = {4},
	journal = {Mathematische Annalen},
	author = {Krause, Henning},
	year = {2008},
	pages = {733--747},
}

@article{hall-rydh_2017,
	title = {The telescope conjecture for algebraic stacks},
	volume = {10},
	url = {https://urn.kb.se/resolve?urn=urn:nbn:se:kth:diva-212585},
	number = {3},
	journal = {Journal of Topology},
	author = {Hall, Jack and Rydh, David},
	year = {2017},
	pages = {776--794},
}

@article{gabriel_1962,
	title = {Des catégories abéliennes},
	volume = {90},
	issn = {2102-622X},
	url = {http://www.numdam.org/item/?id=BSMF_1962__90__323_0},
	doi = {10.24033/bsmf.1583},
	journal = {Bulletin de la Société Mathématique de France},
	author = {Gabriel, Pierre},
	year = {1962},
	pages = {323--448},
}

@article{stevenson_2013,
	title = {Support theory via actions of tensor triangulated categories},
	volume = {2013},
	issn = {1435-5345},
	doi = {10.1515/crelle-2012-0025},
	number = {681},
	journal = {Journal für die reine und angewandte Mathematik},
	author = {Stevenson, Greg},
	year = {2013},
	pages = {219--254},
}

@article{tarrio-lopez-salorio_2004,
	title = {Bousfield localization on formal schemes},
	volume = {278},
	issn = {0021-8693},
	doi = {10.1016/j.jalgebra.2004.02.030},
	number = {2},
	journal = {Journal of Algebra},
	author = {Alonso Tarrío, Leovigildo and Jeremías López, Ana and Souto Salorio, María José},
	year = {2004},
	pages = {585--610},
}

@article{kanda_2015,
	title = {Classification of categorical subspaces of locally noetherian schemes},
	volume = {20},
	issn = {1431-0635},
	doi = {10.4171/dm/522},
	journal = {Documenta Mathematica},
	author = {Kanda, Ryo},
	year = {2015},
	pages = {1403--1465},
}

@article{zhang-cai_2017,
	title = {A note on thick subcategories and wide subcategories},
	volume = {19},
	issn = {15320073, 15320081},
	doi = {10.4310/HHA.2017.v19.n2.a8},
	number = {2},
	journal = {Homology, Homotopy and Applications},
	author = {Zhang, Chao and Cai, Hongyan},
	year = {2017},
	pages = {131--139},
}

@misc{hennion-porta-vezzosi_2016,
      title={Formal gluing along non-linear flags}, 
      author={Benjamin Hennion and Mauro Porta and Gabriele Vezzosi},
      year={2016},
      eprint={1607.04503},
      archivePrefix={arXiv},
      primaryClass={math.AG},
}

@article{antieau-gepner-heller_2019,
  author            = {Benjamin Antieau and David Gepner and Jeremiah Heller},
  journal           = {Inventiones Mathematicae},
  title             = {$K$-theoretic obstructions to bounded $t$-structures},
  year              = {2019},
  volume            = {216},
  pages             = {241--300},
}

@article{takahashi_2009,
  author            = {Ryo Takahashi},
  journal           = {Kyoto Journal of Mathematics},
  title             = {On localizing subcategories of derived categories},
  year              = {2009},
  volume            = {49},
  pages             = {771--783},
}

@article{aambo_2024,
  title         = {Algebraicity in monochromatic homotopy theory},
  author        = {Torgeir Aambø},
  year          = {2024},
  journal       = {To appear in Algebraic \& Geometric topology},
  note          = {arXiv:2403.07725},

}

@misc{antieau_2021,
  title={On the uniqueness of infinity-categorical enhancements of triangulated categories}, 
  author={Benjamin Antieau},
  year={2021},
  eprint={1812.01526},
  archivePrefix={arXiv},
  primaryClass={math.AG}
}

@article{beilinson-bernstein-deligne_1983,
  author            = {Alexander Beilinson and Joseph Bernstein and Pierre Deligne},
  journal          = {Astérisque},
  title             = {Faisceaux pervers},
  year              = {1982},
  volume = {100}
}

@article{barthel-heard-valenzuela_2018,
  title   = {Local duality in algebra and topology},
  volume  = {335},
  issn    = {0001-8708},
  url     = {https://www.sciencedirect.com/science/article/pii/S0001870818302652},
  doi     = {10.1016/j.aim.2018.07.017},
  journal = {Advances in Mathematics},
  author  = {Tobias Barthel and Drew Heard and Gabriel Valenzuela},
  year    = {2018},
  pages   = {563--663}
}

@article{barthel-heard-valenzuela_2020,
  title   = {Derived completion for comodules},
  volume  = {161},
  issn    = {0025-2611, 1432-1785},
  url     = {http://link.springer.com/10.1007/s00229-018-1094-0},
  doi     = {10.1007/s00229-018-1094-0},
  number  = {3-4},
  journal = {Manuscripta Mathematica},
  author  = {Tobias Barthel and Drew Heard and Gabriel Valenzuela},
  year    = {2020},
  pages   = {409--438}
}

@article{greenlees-may_92,
  title   = {Derived functors of ${I}$-adic completion and local homology},
  volume  = {149},
  issn    = {0021-8693},
  url     = {https://www.sciencedirect.com/science/article/pii/002186939290026I},
  doi     = {10.1016/0021-8693(92)90026-I},
  number  = {2},
  journal = {Journal of Algebra},
  author  = {Greenlees, John P. C. and May, Jon P.},
  year    = {1992},
  pages   = {438--453}
}

@book{lurie_09,
  title     = {Higher topos theory},
  note      = {ISBN: 978-0-691-14049-0},
  url       = {https://www.jstor.org/stable/j.ctt7s47v},
  publisher = {Princeton University Press},
  author    = {Lurie, Jacob},
  year      = {2009}
}

@misc{Lurie_HA,
  author = {Jacob Lurie},
  title  = {Higher algebra},
  year   = {2017},
  note   = {Available at \href{https://www.math.ias.edu/~lurie/papers/HA.pdf}{the authors website}.}
}

@book{lurie_SAG,
  author = {Jacob Lurie},
  date   = {},
  title  = {Spectral algebraic geometry},
  year   = {2016},
  note   = {Available at \href{https://www.math.ias.edu/~lurie/papers/SAG-rootfile.pdf}{the authors website}.}
}

@article{pstragowski_2022,
  title   = {Synthetic spectra and the cellular motivic category},
  volume  = {232},
  issn    = {1432-1297},
  url     = {https://doi.org/10.1007/s00222-022-01173-2},
  doi     = {10.1007/s00222-022-01173-2},
  number  = {2},
  journal = {Inventiones Mathematicae},
  author  = {Pstr\a{}gowski, Piotr},
  year    = {2023},
  pages   = {553--681}
}

\textbf{Torgeir Aamb\o:} Department of Mathematical Sciences, Norwegian University of Science and Technology, Trondheim \\
\textbf{Email address:} torgeir.aambo@ntnu.no \\
\textbf{Website:} \href{https://folk.ntnu.no/torgeaam/ }{https://folk.ntnu.no/torgeaam/}  
\end{document}